\documentclass[a4paper]{article}
\usepackage[english]{babel}
\usepackage[utf8]{inputenc}
\RequirePackage{amsmath,amsthm,amssymb,enumitem,graphicx,ifpdf}
\RequirePackage[left=1in,right=1in,top=1in,bottom=1in]{geometry}
\newtheorem{theorem}{Theorem}[subsection]

\newtheorem*{claim*}{Claim}
\newtheorem*{theorem*}{Theorem}

\newtheorem*{corollary*}{Corollary}

\newtheorem*{lemma*}{Lemma}
\theoremstyle{definition}

\newtheorem*{definition*}{Definition}

\newtheorem*{proposition*}{Proposition}

\newtheorem*{notation*}{Notation}

\newtheorem*{remark*}{Remark}

\numberwithin{equation}{section}
\newcommand{\R}{\mathbb{R}} 
\newcommand{\C}{\mathbb{C}}

\newcommand{\N}{\mathbb{N}}

\newcommand{\Li}{\text{Li}}

\newcommand{\Beta}{\text{B}}

\usepackage{cite}

\title{On Central Binomial Series Related to $\zeta(4).$}
\author{Vivek Kaushik}
\begin{document}
\maketitle

\begin{abstract}
In this paper, we prove two related central binomial series identities: $B(4)=\sum_{n \geq 0} \frac{\binom{2n}n}{2^{4n}(2n+1)^3}=\frac{7 \pi^3}{216}$ and 
$C(4)=\sum_{n \in \mathbb{N}} \frac{1}{n^4 \binom{2n}n}=\frac{17 \pi^4}{3240}.$ Both series resist all the standard approaches used to evaluate other well-known series. To prove the first series identity, we will evaluate a log-sine integral that is equal to $B(4).$ Evaluating this log-sine integral will lead us to computing closed forms of polylogarithms evaluated at certain complex exponentials. To prove the second identity, we will evaluate a double integral that is equal to $C(4).$ Evaluating this double integral will lead us to computing several polylogarithmic integrals, one of which has a closed form that is a linear combination of $B(4)$ and $C(4).$ After proving these series identities, we evaluate several challenging logarithmic and polylogarithmic integrals, whose evaluations involve surprising appearances of integral representations of $B(4)$ and $C(4).$ We also provide an insight into the generalization of a modern double integral proof of Euler's celebrated identity $\sum_{n \in \N} \frac{1}{n^2}=\frac{\pi^2}{6},$ in which we encounter an integral representation of $C(4).$ 
\end{abstract}

\section{Introduction}
Consider the series 
\begin{align}
    B(4) &= \sum_{n \geq 0} \frac{\binom{2n}n}{2^{4n}(2n+1)^3} \label{B(4)} \\
    C(4) &=\sum_{n \in \N} \frac{1}{n^4\binom{2n}n} \label{C(4)}.
\end{align}
Both series are known as central binomial series due to the appearance of the central binomial coefficient $\binom{2n}n$ in each summand. The values of each sum are
\begin{align}
B(4) &= \frac{7 \pi^3}{216}  \label{B(4) Identity} \\ 
C(4) &= \frac{17 \pi^4}{3240} \label{C(4) Identity},
\end{align} 
but it is not known who first discovered these results. van Der Poorten \cite{AVP} numerically conjectured \eqref{C(4) Identity}, but the literature suggests either Comtet \cite{LC} or Lewin \cite{LL} had discovered \eqref{C(4) Identity} earlier. 

Both $B(4)$ and $C(4)$ resist many of the standard techniques used to evaluate other series such as
$$\sum_{n \in \N} \frac{1}{n^{2k}}$$ for $k \in \N.$ Such techniques include multiple integration, Fourier Series, and the Calculus of Residues (see all of \cite{DK,BCK,BFY, JC, NE, VK}).  Many authors, however, have studied intimate connections between central binomial series and
log-sine integrals (see all of \cite{JC2009, JC2011, BS, NK, IJZ, ll1958, FB, Borwein}). In particular, $B(4)$ and $C(4)$ are equal to the integrals 
\begin{align}
    I &= \int_{0}^{\frac{\pi}{6}} \log^2(2 \sin(t)) \ dt \label{B(4) Log-Sine Integral} \\
    J &= \int_{0}^{\frac{\pi}{6}} 8t\log^2(2 \sin(t)) \ dt \label{C(4) Log-Sine Integral},
\end{align}
respectively. To evaluate either integral, one must rewrite $2 \sin(t)$ in terms of complex exponentials, expand out the integrand, and evaluate several logarithmic integrals. These integrals have closed forms in terms of polylogarithms, which are complex infinite series of the form
$$\Li_k(z)= \sum_{n \in \N} \frac{z^n}{n^k},$$
where $k \in \N$ and $z \in \C$ satisfies $|z| \leq 1.$ In particular, one will have to explicitly compute $\Li_k(e^{i\theta})$ for certain $\theta \in \R,$ extract real parts, and sum all such results together. This procedure is manageable for evaluating the integral $I,$ but laborious for evaluating the integral $J.$

In this paper, we will prove the central binomial series identities \eqref{B(4) Identity} and \eqref{C(4) Identity}. To prove \eqref{B(4) Identity}, we will first show that the series $B(4)$ is equal to $I.$ To do this, we will change variables $x=2 \sin(t)$ on $I,$ convert the transformed integrand into a binomial series, interchange sum and integral, and finally integrate term-by-term to recover $B(4).$ Then, we will show $I=7 \pi^3/216$ using the procedure described in the previous paragraph to evaluate $I.$

Next, we will prove \eqref{C(4) Identity} by showing that the series $C(4)$ is equal to the double integral
\begin{align}
K &= \int_{0}^{1}\int_{0}^{1-x} \frac{2 \log^2(x+y)}{1-xy} \ dy \ dx
\end{align}  
and evaluating $K$ directly. To show $C(4)$ is equal to $K,$ we will change variables $x=(u+v)/2, \ y=(u-v)/2$ on $K,$ integrate with respect to $v,$ convert the resulting integrand into a special binomial series, interchange sum and integral, and finally integrate term-by-term to recover $C(4).$ Then, we will show $K=17 \pi^4/3240$ by integrating the original double integral representation with respect to $y,$ which will result in us evaluating three different polylogarithmic integrals. Two of these integrals have closed forms that are rational multiples of $\pi^4,$ while the remaining integral has a closed form that is a linear combination of $B(4)$ and $C(4).$ Our proof of \eqref{C(4) Identity} does require elaborate integral calculations, but it is less laborious than the standard approach of evaluating the log-sine integral $J.$ 

After proving the central binomial series identities, we state and evaluate some difficult logarithmic and polylogarithmic integrals. In all of their evaluations, certain integral representations of $B(4)$ and $C(4)$ make surprising appearances in all of these integral evaluations. We conclude with a remark about the generalization of Apostol's famous double integration solution \cite{TA} to proving $$\sum_{n \in \N} \frac{1}{n^2}=\frac{\pi^2}6.$$ In particular, in mimicking Apostol's method to prove 
$$\sum_{n \in \N}\frac{1}{n^4}=\frac{\pi^4}{90},$$ we will encounter an integral representation of $C(4).$

\section{Preliminaries}
This section will be dedicated to recalling standard definitions and results from complex variables and analytic number theory that we will need to prove our main central binomial series identities.

\subsection{Complex Numbers}
Recall a \textbf{complex number} $z \in \C$ is of the form
\begin{align} \label{Complex Number Definition}
z=x+iy,
\end{align}
where $x,y \in \R$ and  $i=\sqrt{-1}.$ 
We call $x$ and $y$ the \textbf{real} and \textbf{imaginary} parts of $z,$ respectively, and write $x=\Re(z)$ and $y=\Im(z).$

If $z \in \C.$ we define its \textbf{modulus} to be the real number 
\begin{align}
|z|=\sqrt{\Re(z)^2+\Im(z)^2}, \label{Modulus Definition}
\end{align}
and if $z \neq 0,$ we define its \textbf{argument} to be the real number 
\begin{align}
\arg(z)=\tan^{-1}\left( \frac{\Im(z)}{\Re(z)} \right) \in [-\pi , \pi), \label{Argument Definition}
\end{align}
and its \textbf{complex logarithm} to be the complex number
\begin{align}
\log(z)=\ln(|z|)+ i \arg(z). \label{Complex Logarithm Definition}
\end{align}

We define a \textbf{complex exponential} to be a complex number of the form 
\begin{align}
e^{i \theta}=\cos(\theta)+i \sin(\theta), \label{Complex Exponential Definition}
\end{align}
where $\theta \in \R.$ 
With \eqref{Complex Exponential Definition}, we can easily deduce the following theorem.
\begin{theorem}[Euler's Formulas]\label{Euler's Formulas}
We have 
\begin{align}
    \cos(\theta) &= \Re(e^{i\theta}) = \frac{e^{i \theta}+e^{-i \theta}}{2} \label{Cosine Complex Exponential Definition} \\
    \sin(\theta)&= \Im(e^{i\theta})= \frac{e^{i \theta}-e^{-i \theta}}{2i} \label{Sine Complex Exponential Definition}.
\end{align}
\end{theorem}

Using \eqref{Sine Complex Exponential Definition}, we can derive another representation of the function $\log(2\sin(t))$ for $t \in (0,\pi)$ that we will need to complete the proof of our first central binomial series identity.
\begin{theorem}[Log-Sine Expansion]\label{Log-Sine Expansion}
If $t \in (0, \pi),$ we have 
\begin{align}
\log(2 \sin(t)) &= \log(1-e^{2it})+ \left(\frac{\pi}{2}-t \right) i. \label{Log-Sine Expansion Formula}
\end{align}
In particular, 
\begin{align}
\Im(\log(2\sin(t)) &=0. \label{Imaginary Part Log-Sine}
\end{align}
\end{theorem}

\begin{proof}
Rewriting $2\sin(t)$ according to the complex exponential formula from \eqref{Sine Complex Exponential Definition} and splitting the logarithm apart, we get 
\begin{align*}
\log(2 \sin(t)) &= \log \left(\frac{e^{it}-e^{-it}}{i} \right) \\
&= \log \left(ie^{-it}(1-e^{2it} \right) ) \\
&= \log(i) + \log(e^{-it}) + \log(1-e^{2it}) \\
&= i \left(\frac{\pi}{2}-t \right) + \log(1-e^{2it}).
\end{align*}

The second statement \eqref{Imaginary Part Log-Sine} is immediate from the fact that $2\sin(t)$ is positive for any $t \in (0,\pi),$ which implies $\log(2\sin(t))$ is real.

\end{proof}

\subsection{Polylogarithms}
The \textbf{polylogarithm} of order $k \in \N$ is defined to be the series
\begin{align}
    \Li_k(z) &= \sum_{n \in \N} \frac{z^n}{n^k}, \label{Polylogarithm Definition}
\end{align}
where $z \in \C$ satisfies $|z| \leq 1.$ 
In particular, when $z=1,$ we recover the well-known \textbf{Zeta function}, which is the series 
\begin{align}
    \zeta(k)&=\Li_k(1)=\sum_{n \in \N} \frac{1}{n^k}, \label{Zeta Function Definition}
\end{align}
and when $z=-1,$ we recover the well-known \textbf{Eta function}, which is the series 
\begin{align}
    \eta(k)&=\Li_k(-1)=\sum_{n \in \N} \frac{(-1)^n}{n^k}. \label{Eta Function Definition}
\end{align}

From \eqref{Polylogarithm Definition}, we can deduce the following functional identities. 
\begin{theorem}[Polylogarithmic Functional Identities]\label{Polylogarithmic Functional Identities}

We have 
\begin{align} 
    \frac{d}{dz} \Li_k(z) &= \frac{\Li_{k-1}(z)}{z} \label{Polylogarithmic Differentiation Formula} \\
    \Li_{k+1}(z) &= \int \frac{\Li_k(t)}{t} \ dt \label{Polylogarithmic Integration Formula} \\
\Li_k(-z) &= \frac{\Li_k(z^2)}{2^{k-1}}-\Li_k(z). \label{Polylogarithmic Negative Argument Formula}
\end{align}
\end{theorem}
The differentiation formula in \eqref{Polylogarithmic Differentiation Formula} and the integration formula in \eqref{Polylogarithmic Integration Formula} are ones we will be repeatedly using throughout this paper for elaborate integral calculations. 

The following theorem highlights a handy algebraic relation between $\eta(k)$ and $\zeta(k).$
\begin{theorem}[Eta and Zeta Function Relation]\label{Eta and Zeta Function Relation}
We have
\begin{align}\label{Eta(k) from Zeta(k)}
\eta(k) &= - \left(1- \frac{1}{2^{k-1}} \right) \zeta(k).
\end{align}
\end{theorem}
\begin{proof}
This follows from substituting $z=1$ into \eqref{Polylogarithmic Negative Argument Formula}. 
\end{proof}

The polylogarithm has some closed forms for certain complex $z$ in the cases $k=1$ and $2.$ We list the most important closed forms that we need for this paper. 

\begin{theorem}[Polylogarithmic Closed Forms]\label{Polylogarithmic Closed Forms}
We have 
\begin{align}
\Li_1(z) &= -\log(1-z), \quad z \neq 1 \label{Li_1(z) Closed Form} \\
\Li_2(1) &= \zeta(2) = \frac{\pi^2}{6} \label{Zeta(2) Closed Form} \\
\Li_2 \left(\frac{1}{2}\right)  &= \frac{\zeta(2)}{2}-\frac{\log^2(2)}{2} \label{Li_2(1/2) Closed Form} \\
\Li_4(1) &= \zeta(4) = \frac{\pi^4}{90}. \label{Zeta(4) Closed Form} 
\end{align}
\end{theorem}
\begin{proof}
To prove \eqref{Li_1(z) Closed Form},  we evaluate the integral
\begin{align}\label{Li_1(z) Integral}
\int_0^z \frac{1}{1-t} \ dt
\end{align}
in two ways. On one hand, directly integrating \eqref{Li_1(z) Integral} with elementary calculus gives
\begin{align*}
\int_0^z \frac{1}{1-t} \ dt &= -\log(1-z).
\end{align*}
On the other hand, we rewrite the integrand of \eqref{Li_1(z) Integral} as a geometric series, and we get
\begin{align}
\int_0^z \frac{1}{1-t} \ dt &= \int_0^z \sum_{n \in \N} t^{n-1} \ dt \nonumber \\
&= \sum_{n \in \N} \int_0^z  t^{n-1} \ dt \nonumber \\
&= \sum_{n \in \N} \frac{z^n}{n}=\Li_1(z). \nonumber
\end{align}
Hence, equating these two representations of \eqref{Li_1(z) Integral} yields \eqref{Li_1(z) Closed Form}.

To prove \eqref{Li_2(1/2) Closed Form}, we evaluate the integral 
\begin{align}\label{Li_2(1/2) Integral}
\int_{0}^{\frac{1}2} -\frac{\log(1-t)}{t} \ dt
\end{align}
in two ways. Notice we can rewrite the numerator of the integrand using  \eqref{Li_1(z) Closed Form}, and using the integration formula in \eqref{Polylogarithmic Integration Formula}, we get
\begin{align*}
\int_{0}^{\frac{1}2} -\frac{\log(1-t)}{t} \ dt &= \int_{0}^{\frac{1}2} \frac{\Li(t)}{t} \ dt=\Li_2\left(\frac{1}{2} \right).
\end{align*}
On the other hand, we integrate \eqref{Li_2(1/2) Integral} by parts. Let $u=-\log(1-t)$ and $dv=\frac{dt}{t}.$ Then, $v=\log(t)$ and $du=\frac{dt}{1-t},$ and we get 
\begin{align}
\int_{0}^{\frac{1}2} -\frac{\log(1-t)}{t} \ dt &= -\log^2\left(\frac{1}{2} \right) - \int_{0}^{\frac{1}{2}} \frac{\log(t)}{1-t} \ dt \nonumber \\
&= -\log^2\left(\frac{1}{2} \right) + \int_{0}^{\frac{1}{2}} \frac{\Li_1(1-t)}{1-t} \ dt \label{Li_2(1/2) Integration by Parts} \\ 
&= -\log^2\left(\frac{1}{2} \right) + \int_{\frac{1}{2}}^{1} \frac{\Li_1(u)}{u} \ du \label{Li_2(1/2) Substitution t=1-u} \\
&= -\log^2\left(2 \right) + \zeta(2) - \Li_2 \left( \frac{1}{2} \right)\nonumber,
\end{align}
where \eqref{Li_2(1/2) Substitution t=1-u} follows from substituting $t=1-u$ into the integral term in \eqref{Li_2(1/2) Integration by Parts}. Hence, equating these two representations of \eqref{Li_2(1/2) Integral} and rearranging terms yields \eqref{Li_2(1/2) Closed Form}.

The proofs of \eqref{Zeta(2) Closed Form} and \eqref{Zeta(4) Closed Form} are well-known results, for which one can consult any one of \cite{TA, VK,  DR, JH, LP, BFY, BCK, NE, DA'K}.
\end{proof}

\begin{remark*}
The identity \eqref{Zeta(2) Closed Form} is commonly known as the \textbf{Basel Problem} in Analytic Number Theory.
\end{remark*}
See \cite{LL} or \cite{GH} for an extensive list of polylogarithmic identities.

We will need to explicitly evaluate $\Li_k(z)$ when $z$ is a complex exponential. To do this, we identify the real and imaginary parts of $\Li_k(e^{i \theta}),$ which we call the \textbf{Clausen Functions}. In particular, the \textbf{Clausen Cosine} and \textbf{Clausen Sine} functions of order $k \in \N,$ are respectively defined to be the series
\begin{align}
C_k(\theta) &= \Re(\Li_k(e^{i\theta}))=\sum_{n \in \N} \frac{\cos(n \theta)}{n^k} \label{Clausen Cosine Definition} \\
S_k(\theta) &= \Im(\Li_k(e^{i\theta}))=\sum_{n \in \N} \frac{\sin(n \theta)}{n^k}, \label{Clausen Sine Definition}
\end{align}
where again, $\theta \in \R.$

We can explicitly evaluate these Clausen functions for certain values of $\theta$ very easily.

\begin{theorem}[Clausen Closed Formulas]\label{Clausen Closed Formulas}
For $t \in (0,\pi),$ we have 
\begin{align}
S_1(2t) &= \frac{\pi}{2}-t \label{S_1(2t) Closed Formula} \\
C_2(2t) &= \zeta(2)- \pi t +t^2 \label{C_2(2t) Closed Formula} \\
S_3(2t) &= \frac{\pi^2 t}{3}-\pi  t^2+\frac{2t^3}{3} \label{S_3(2t) Closed Formula} \\
C_4(2t) &= \zeta(4)-\frac{\pi ^2t^2}{3}+\frac{2 \pi t^3}{3}-\frac{t^4}{3}. \label{C_4(2t) Closed Formula} 
\end{align} 
\end{theorem}

\begin{proof}
On one hand, recall from \eqref{Imaginary Part Log-Sine} that $ \Im (\log(2\sin(t)))=0.$ On the other hand, recalling our representation of $\log(2\sin(t))$ from \eqref{Log-Sine Expansion Formula} and the closed form for $\Li_1(z)$ from \eqref{Li_1(z) Closed Form}, we have
\begin{align*}
\Im (\log(2\sin(t))) &= \frac{\pi}{2}-t - \Im(\Li_1(e^{2it})) \\
&=\frac{\pi}{2}-t - S_1(2t).
\end{align*} 
Equating our two representations for $\Im (\log(2\sin(t)))$ and rearranging terms proves the first statement \eqref{S_1(2t) Closed Formula}. 

To obtain the second statement in \eqref{C_2(2t) Closed Formula}, we compute the integral 
\begin{align}\label{C_2(2t) Integral}
\int_0^t S_1(2u) \ du
\end{align}
in two ways. On one hand using the original definition of $S_1(2u)$ by substituting $\theta=2u$ and $k=1$ into \eqref{Clausen Sine Definition}, we have 
\begin{align*}
\int_0^t S_1(2u) \ du  &= \int_0^t \sum_{n \in \N} \frac{\sin(2nu)}{n} \ du \\
&= \sum_{n \in \N} \int_0^t  \frac{\sin(2nu)}{n} \ du \\
&= \sum_{n \in \N} \frac{1-\cos(2nt)}{2n^2} = \frac{\zeta(2)-C_2(2t)}{2}.
\end{align*}
On the other hand, using \eqref{S_1(2t) Closed Formula}, we have
\begin{align*}
\int_0^t S_1(2u) \ du  &= \int_0^t \frac{\pi}{2} - u \ du \\
&= \frac{\pi t}{2}-\frac{t^2}2.
\end{align*} Equating both these representations of \eqref{C_2(2t) Integral} and solving for $C_2(2t)$ will yield \eqref{C_2(2t) Closed Formula}.

Mimic the same argument from above by computing the integrals $\int_0^t C_2(2u) \ du$ and $\int_0^t S_3(2u) \ du$  in two ways to prove \eqref{S_3(2t) Closed Formula} and \eqref{C_4(2t) Closed Formula}, respectively.
\end{proof}

\begin{remark*}
We can in fact rederive $\zeta(2)=\pi^2/6$ by substituting $t=\pi/2$ into the formula for $C_2(2t)$ in \eqref{C_2(2t) Closed Formula}. The left hand side of \eqref{C_2(2t) Closed Formula} is $\eta(2),$ while the right hand side of \eqref{C_2(2t) Closed Formula} is $\zeta(2) - \pi^2/4.$ The result follows from rewriting $\eta(2)=-\zeta(2)/2$ by virtue of substituting $k=2$ into the algebraic formula for $\eta(k)$ in \eqref{Eta(k) from Zeta(k)} and rearranging terms. By mimicking this exact same argument by substituting $t=\pi/2$ into the formula for $C_4(2t)$ in \eqref{C_4(2t) Closed Formula}, we can recover $\zeta(4)=\pi^4/90.$ 
\end{remark*}

\subsection{Beta Functions and Binomial Coefficients}
We now recall the well-known Beta function and generalized binomial coefficients, which will appear all throughout the second central binomial series identity proof.

The \textbf{Beta function} is defined to be any one of the following two integrals
\begin{align} 
\Beta(u,v) &= \int_{0}^{1} x^{u-1} (1-x)^{v-1} \ dx= \int_{0}^{\infty} \frac{t^{u-1}}{(1+t)^{u+v}} \ dt \label{Beta Function Definitions} ,
\end{align}
where $u,v \in \C$ with $\Re(u),\Re(v)>0.$ Note that the integrals in \eqref{Beta Function Definitions} are equal under the change of variables $x=t/(1+t).$ 

We define the \textbf{generalized binomial coefficient} to be the number
\begin{align} \label{General Binomial Coefficient Definition}
\binom{\alpha}{k} &= \frac{\alpha(\alpha-1) \ \dots \ (\alpha-k+1)}{k!} ,
\end{align}
where $\alpha \in \C$ and $k \in \N \cup \lbrace 0 \rbrace.$
In the case $\alpha=2k,$ we call $\binom{2k}k$ the \textbf{central binomial coefficient}.

There are some identities specifically involving the central binomial coefficient which we will need.

\begin{theorem}[Generalized Binomial Coefficient Identities]\label{Generalized Binomial Coefficient Identities}
For $n \in \N \cup \lbrace 0 \rbrace,$ we have
\begin{align} 
\binom{-\frac{1}{2}}{n} &= (-1)^n \left(\frac{1}{2} \right)^{2n} \binom{2n}n \label{-1/2 Choose n formula} \\
\Beta \left(\frac{2n+1}2,\frac{2n+1}2 \right) &=\frac{\pi}{2^{4 n}} \binom{2 n}{n} \label{Beta(2n+1/2,2n+1/2) formula} \\
\Beta(n,n) &= \frac{2}{n\binom{2n}n} , \quad n >0. \label{Beta(n,n) Formula}
\end{align}
\end{theorem}
 
We recall the well-known binomial series expansion, which we will use to prove our first central binomial series identity.
\begin{theorem}[Binomial Series Expansion]\label{Binomial Series Expansion}
If $\alpha \in \C$ and $x \in \R$ satisfies  $|x| < 1,$ we have
\begin{align}\label{Binomial Series Formula}
(1+x)^{\alpha} &= \sum_{n \geq 0} \binom{\alpha}n x^n.
\end{align} 
\end{theorem}

The next binomial series representation involves the arcsine function, which we will use to prove our second central binomial series identity. 
\begin{theorem}[ArcSine Series Expansion]\label{ArcSine Series Expansion}
If $x \in \R$ satisfies $|x|<2,$ we have 
\begin{align}
\frac{\sin^{-1} \left(\frac{x}{2} \right)}{\sqrt{4-x^2}} &= \sum_{n \in \N} \frac{x^{2 n-1}}{2n \binom{2 n}{n}}. \label{ArcSine Squared Formula}
\end{align} 
\end{theorem}

\begin{proof}
We evaluate the integral
\begin{align} \label{ArcSine Original Integral}
\int_{0}^{1} \frac{x}{4-4x^2(y-y^2)} \ dy
\end{align}
in two ways.

On one hand, we can complete the square in the denominator of the integrand and evaluate \eqref{ArcSine Original Integral} directly as follows:
\begin{align}
\int_{0}^{1} \frac{x}{4-4x^2(y-y^2)} \ dy &= \int_{0}^{1} \frac{1}{4x \left(\frac{4-x^2}{4 x^2}+\left(y-\frac{1}{2} \right)^2 \right)} \ dy \label{Arcsine Complete the Square} \\
&= \int_{-\frac{1}{2}}^{\frac{1}{2}} \frac{1}{4x \left(\frac{4-x^2}{4 x^2}+u^2 \right)} \ du \label{Arcsine Substitute u=y-1/2} \\
&= \frac{\tan ^{-1}\left(\frac{x}{\sqrt{4-x^2}}\right)}{\sqrt{4-x^2}}= \frac{\sin ^{-1}\left(\frac{x}{2}\right)}{\sqrt{4-x^2}} \nonumber, 
\end{align}
where we substituted $u=y-1/2$ into the right hand side of \eqref{Arcsine Complete the Square} to get \eqref{Arcsine Substitute u=y-1/2}. On the other hand, we evaluate \eqref{ArcSine Original Integral} by rewriting the integrand as a geometric series, in which we get
\begin{align}
\int_{0}^{1} \frac{x}{4-4x^2(y-y^2)} \ dy &= \int_0^1  \sum_{n \in \N} \frac{x^{2n-1} y^{n-1} (1-y)^{n-1}}4  \ dy \nonumber \\
&= \sum_{n \in \N} \int_0^1 \frac{x^{2n-1} y^{n-1} (1-y)^{n-1}}4 \ dy \label{ArcSine Monotone Convergence Theorem} \\
&= \sum_{n \in \N} \frac{x^{2n-1} \Beta(n,n)}4  =  \sum_{n \in \N} \frac{x^{2n-1}}{2n \binom{2n}n} \label{ArcSine Beta(n,n) Formula},
\end{align}
where the interchanging of sum and integral in \eqref{ArcSine Monotone Convergence Theorem} is permitted by the Monotone Convergence theorem, and \eqref{ArcSine Beta(n,n) Formula} follows from simplifying the summand using the formula for $\Beta(n,n)$ from \eqref{Beta(n,n) Formula}. Thus, equating these two representations for \eqref{ArcSine Original Integral} gives the desired result.
\end{proof}

\section{Evaluating $B(4)$ with a Log-Sine Integral}
We are now ready to prove our first identity \eqref{B(4) Identity}, namely
\begin{align*}
B(4) &= \sum_{n \geq 0} \frac{\binom{2n}n}{2^{4n} (2n+1)^3} = \frac{7 \pi^3}{216}
\end{align*}
using the log-sine integral 
\begin{align} \label{I Original Definition}
I &= \int_{0}^{\frac{\pi}{6}} \log^2(2\sin(t)) \ dt. 
\end{align}

We first establish that $I=B(4).$ 
\begin{theorem}
We have $I=B(4).$ 
\end{theorem}
\begin{proof}
Substituting $x= 2\sin(t)$ into \eqref{I Original Definition}, we get 
\begin{align}
I &= \int_{0}^{1} \frac{\log^2(x)}{2\sqrt{1-\frac{x^2}4}} \ dx \label{I COV} \\
&= \int_{0}^{1} \sum_{n \geq 0} \binom{-\frac{1}{2}}n \left(-\frac{1}{4} \right)^n \frac{x^{2n} \log^2(x)}2 \ dx \label{I COV Binomial Series} \\
&=  \sum_{n \geq 0} \left(\frac{1}{4} \right)^{2n} \binom{-\frac{1}{2}}n  \int_{0}^{1}  \frac{x^{2n} \log^2(x)}2 \ dx\label{I COV Monotone Convergence Theorem} \\
&=\sum_{n \geq 0} \left(\frac{1}{4} \right)^{2n} \binom{-\frac{1}{2}}n  \frac{1}{(2n+1)^3} \label{I COV Integration By Parts} \\
&=\sum_{n \geq 0} \left(\frac{1}{2} \right)^{4n} \binom{2n}n \frac{1}{(2n+1)^3}=B(4), \label{I -1/2 Choose n}
\end{align}
where \eqref{I COV Binomial Series} follows from rewriting the integrand \eqref{I COV} using the binomial series expansion formula in \eqref{Binomial Series Formula} and simplifying the resulting summand, and the interchanging of sum and integral in \eqref{I COV Monotone Convergence Theorem} is permitted by the Monotone Convergence Theorem. Finally, \eqref{I COV Integration By Parts} follows evaluating the integral inside the summand of \eqref{I COV Monotone Convergence Theorem} with repeated integration by parts, and \eqref{I -1/2 Choose n} follows from rewriting and simplifying the summand with the formula for $\binom{-1/2}n$ from \eqref{-1/2 Choose n formula}. 
\end{proof}

We now evaluate $I$ directly. Using our formula for $\log(2\sin(t))$ from \eqref{Log-Sine Expansion Formula} to expand the integrand of \eqref{I Original Definition} and using linearity of the integral, we get
\begin{align}
I &= \int_{0}^{\frac{\pi}{6}}  \left(\log(1-e^{2it})+ \left(\frac{\pi}{2}-t \right) i \right)^2 \ dt \nonumber \\
&=  \int_{0}^{\frac{\pi}{6}} -\left(\frac{\pi}{2}-t \right)^2 \ dt + \int_{0}^{\frac{\pi}{6}} i (\pi-2t) \log(1-e^{2it}) \ dt  + \int_{0}^{\frac{\pi}{6}}  \log^2(1-e^{2it}) \ dt \label{I Expanded Out}. 
\end{align}

We denote the three integrals in \eqref{I Expanded Out} as:
\begin{align}
I_1 &= \int_{0}^{\frac{\pi}{6}} -\left(\frac{\pi}{2}-t \right)^2 \ dt \label{I_1 Original Definition} \\
I_2 &= \int_{0}^{\frac{\pi}{6}} i (\pi-2t) \log(1-e^{2it}) \ dt \label{I_2 Original Definition} \\
I_3 &= \int_{0}^{\frac{\pi}{6}}  \log^2(1-e^{2it}) \ dt\label{I_3 Original Definition}.
\end{align}
Since $I$ is a real number, it suffices to compute just the real parts of $I_1, I_2, I_3$ and sum those results together. 

We will start with $I_1,$ which is easy to evaluate with elementary calculus.
\begin{theorem}
We have 
\begin{align}\label{Value I_1}
\Re(I_1)&=I_1=-\frac{19 \pi^3}{648}.
\end{align} 
\end{theorem}
\begin{proof}
Substituting $x=\frac{\pi}{2}-t$ into \eqref{I_1 Original Definition}, we have 
\begin{align*}
I_1=\int_{\frac{\pi}{2}}^{\frac{\pi}{3}} x^2 \ dx &=-\frac{19 \pi^3}{648}.
\end{align*} 
\end{proof}

Using the Clausen Function closed formulas from  \textbf{Theorem \ref{Clausen Closed Formulas}}, we can obtain $\Re(I_2).$ 
\begin{theorem}
We have 
\begin{align}\label{Value I_2}
\Re(I_2)&=\frac{19 \pi^3}{324}.
\end{align} 
\end{theorem}
\begin{proof}
We compute the integral of the real part of the integrand in \eqref{I_2 Original Definition} and see
\begin{align}
\Re(I_2) &= \int_{0}^{\frac{\pi}{6}} \Re \left(i (\pi-2t) \log(1-e^{2it}) \right) \ dt \nonumber \\
&= \int_{0}^{\frac{\pi}{6}}  \Re \left(-i (\pi-2t) \Li_1(e^{2it}) \right)  \ dt \label{I_2 Li_1(e^2it) Formula}\\
&= \int_{0}^{\frac{\pi}{6}}(\pi -2t)S_1(2t) \ dt \label{I_2 Real Part Transformed} \\
&= \int_{0}^{\frac{\pi}{6}}2\left(\frac{\pi}{2} -t \right)^2  \ dt \label{I_2 S_1(2t) Formula}\\
&=-2I_1=\frac{19 \pi^3}{324} \nonumber,
\end{align}
where \eqref{I_2 Li_1(e^2it) Formula} follows from substituting $z=e^{2it}$ into the closed form of $\Li_1(z)$ from \eqref{Li_1(z) Closed Form}, and \eqref{I_2 S_1(2t) Formula} follows from rewriting and the integrand of \eqref{I_2 Real Part Transformed} with the closed form of $S_1(2t)$ from \eqref{S_1(2t) Closed Formula}.
\end{proof}

Now, we need to evaluate $I_3$ in order to obtain $\Re(I_3).$  Before we proceed to evaluate $I_3,$ we need to use the following antiderivative identity, whose proof will require repeated integration by parts.

\begin{theorem}
For $z \in \C$ with $0<|z|<1,$ we have 
\begin{align} \label{log^2(1-z)/z Antiderivative Formula}
\int \frac{\log^2(1-z)}{z} \ dz &= \log (z) \log ^2(1-z)+2 \text{Li}_2(1-z) \log (1-z)-2 \text{Li}_3(1-z).
\end{align}
\end{theorem} 

\begin{proof}
Notice that the left hand side of \eqref{log^2(1-z)/z Antiderivative Formula} is equal to 
\begin{align}\label{log^2(1-z)/z Original Integral}
\int \frac{\Li_1(z) {}^2}{z} \ dz.
\end{align}
by \eqref{Li_1(z) Closed Form}. We now integrate \eqref{log^2(1-z)/z Original Integral} by parts. Let $u=\Li_1(z) {}^2$ and $dv=\frac{dz}{z}.$ Then, $v=\log(z)$ and applying the differentiation formula from \eqref{Polylogarithmic Differentiation Formula} together with the Chain Rule, we get $du = - \frac{2 \Li_1(z)}{1-z} \ dz.$ Thus,
\begin{align}
\int \frac{\Li_1(z) {}^2}{z} \ dz &= \Li_1(z) {}^2 \log(z) - \int \frac{2 \log(z) \Li_1(z)}{1-z} \ dz \nonumber \\
&=\Li_1(z) {}^2 \log(z) + \int \frac{2 \Li_1(1-z) \Li_1(z)}{1-z} \ dz. \label{log^2(1-z)/z Integrated By Parts Once}
\end{align}
Now integrating the integral term in \eqref{log^2(1-z)/z Integrated By Parts Once} by parts with $u=2 \Li_1(z)$ and $dv=\frac{\Li_1(1-z)}{1-z} \ dz,$ we have $v=-\Li_2(1-z)$ and  $du=\frac{2}{1-z} \ dz,$ and
\begin{align}
\int \frac{\Li_1(z) {}^2}{z} \ dz &=\Li_1(z) {}^2 \log(z) - 2 \Li_1(z)\Li_2(1-z) + \int \frac{2 \Li_2(1-z) }{1-z} \ dz \nonumber \\
&= \Li_1(z) {}^2 \log(z) - 2 \Li_1(z)\Li_2(1-z)- 2\Li_3(1-z). \label{log^2(1-z)/z Integrated By Parts Twice}
\end{align}
The result follows upon replacing $\Li_1(z)$ with $-\log(1-z)$ in all terms of \eqref{log^2(1-z)/z Integrated By Parts Twice}.
\end{proof}
   
With this antiderivative formula, we can evaluate $I_3$ and extract its real part. 
\begin{theorem}
We have 
\begin{align*}
\Re(I_3) &= \frac{\pi^3}{324}.
\end{align*}
\end{theorem}
\begin{proof}
Substituting $z=e^{2it}$ into \eqref{I_3 Original Definition}, we get 
\begin{align}
I_3 &= \int_{1}^{e^{\frac{i \pi}{3}}} \frac{\log^2(1-z)}{2iz} \ dx \label{I_3 COV}.
\end{align}

Letting $\psi(z)$ be the antiderivative expression on the right hand side of \eqref{log^2(1-z)/z Antiderivative Formula}, we can use the Fundamental Theorem of Calculus to see that
\begin{align}
\Re(I_3) &= \Re \left(\frac{\psi(e^{\frac{i \pi}3})}{2i} \right)-\Re \left(\frac{\lim_{z \rightarrow 1} \psi(z)}{2i} \right) \nonumber \\
&= \Re \left(\frac{\log (e^{\frac{i \pi}{3}}) \log ^2(1-e^{\frac{i \pi}{3}})}{2i}+\frac{2 \text{Li}_2(1-e^{\frac{i \pi}{3}}) \log (1-e^{\frac{i \pi}{3}})}{2i}-\frac{2 \text{Li}_3(1-e^{\frac{i \pi}{3}})}{2i} \right)-0 \\
&= -\frac{\pi^3}{54} - \frac{\pi}{3} C_2 \left(\frac{\pi}{3} \right) + S_3 \left(\frac{\pi}{3} \right) \label{I_3 Clausen} \\ 
&= -\frac{\pi^3}{54} - \frac{\pi^3}{108} + \frac{5 \pi^3}{162}=\frac{\pi^3}{324} \label{I_3 Simplified},
\end{align}
where \eqref{I_3 Simplified} follows from simplifying the second and third terms of \eqref{I_3 Clausen} by substituting $t=\pi/6$ into the formulas for $C_2(2t)$ and $S_3(2t)$ from \eqref{C_2(2t) Closed Formula} and \eqref{S_3(2t) Closed Formula}, respectively. 
\end{proof}

Putting everything together, we have 
\begin{align*}
I &= \Re(I_1)+\Re(I_2)+\Re(I_3) \\
&= -\frac{19 \pi^3}{648}+\frac{19 \pi^3}{324}+ \frac{\pi^3}{324}=\frac{7 \pi^3}{216}.
\end{align*}
Hence,
\begin{align*}
B(4) =\sum_{n \geq 0} \frac{\binom{2n}n}{2^{4n} (2n+1)^3} = \frac{7 \pi^3}{216}.
\end{align*}

\section{Evaluating $C(4)$ with a Double Integral}
We now prove our second identity \eqref{C(4) Identity}, namely
\begin{align*}
C(4) &= \sum_{n \in \N} \frac{1}{n^4 \binom{2n}n}= \frac{17 \pi^4}{3240}, 
\end{align*}
using the double integral 
\begin{align} \label{K Original Definition}
K &= \int_{0}^1 \int_{0}^{1-x} \frac{2 \log^2(x+y)}{1-xy} \ dy \ dx.
\end{align}

We first establish that $K=C(4).$ 
\begin{theorem}
We have $K=C(4).$ 
\end{theorem}
\begin{proof}
Substituting $x=\frac{u+v}{2}$ and $y=\frac{u-v}{2}$ into \eqref{K Original Definition} and using the Change of Variables formula, we get 
\begin{align}
K &= \int_{0}^{1} \int_{-u}^u \frac{4 \log^2(u)}{4-u^2+v^2} \ dv \ du \nonumber  \\
&=  \int_{0}^{1} \frac{8 \tan^{-1} \left(\frac{u}{\sqrt{4-u^2}} \right) \log^2(u)}{\sqrt{4-u^2}} \ du \nonumber \\
&=  \int_{0}^{1} \frac{8 \sin^{-1} \left(\frac{u}{2} \right) \log^2(u)}{\sqrt{4-u^2}} \ du \label{K COV} \\
&=  \int_{0}^{1} \sum_{n \in \N} \frac{4}{n \binom{2 n}{n}} u^{2 n-1} \log^2(u) \ du \label{K ArcSine Series} \\
&=  \sum_{n \in \N} \frac{4}{n \binom{2 n}{n}}\int_{0}^{1}  u^{2 n-1} \log^2(u) \ du \label{K Monotone Convergence Theorem} \\
&= \sum_{n \in \N} \frac{1}{n^4 \binom{2n}n}=C(4) \label{K IBP},
\end{align}
where \eqref{K ArcSine Series} follows from rewriting the integrand in \eqref{K COV} using the arcsine series representation from \eqref{ArcSine Squared Formula} and simplifying the resulting summand, and the interchanging of sum and integral in \eqref{K Monotone Convergence Theorem} is permitted by the Monotone Convergence Theorem. Finally, \eqref{K IBP} follows from repeated integration by parts of the integral inside the sum of \eqref{K Monotone Convergence Theorem} and simplifying the resulting summand. 
\end{proof}

We now evaluate the original double integral representation of $K$ in \eqref{K Original Definition} directly. We make the substitution $y=\frac{1-u(1+x^2)}{x}$ on \eqref{K Original Definition} to see that 
\begin{align}
K &= \int_{0}^{1} \int_{\frac{1}{1+x^2}}^{\frac{1-x+x^2}{1+x^2}} -\frac{2 \log ^2\left(\left(\frac{1+x^2}{x}\right) (1-u) \right)}{u x} \ du \ dx \nonumber \\
&= \int_{0}^{1} \int_{\frac{1}{1+x^2}}^{\frac{1-x+x^2}{1+x^2}} -\frac{2 \left(\log \left(\frac{1+x^2}{x}\right) + \log(1-u) \right)^2}{ux} \ du \ dx \nonumber \\
&=  \int_{0}^{1} \int_{\frac{1}{1+x^2}}^{\frac{1-x+x^2}{1+x^2}} -\frac{2 \log^2 \left(\frac{1+x^2}{x}\right)+4\log \left(\frac{1+x^2}{x}\right)\log(1-u) + 2 \log^2(1-u)}{ux} \ du \ dx \nonumber \\
&= \int_{0}^{1} \int_{\frac{1}{1+x^2}}^{\frac{1-x+x^2}{1+x^2}} -\frac{2 \log^2 \left(\frac{1+x^2}{x}\right)-4\log \left(\frac{1+x^2}{x}\right)\Li_1(u) + 2 \log^2(1-u)}{ux} \ du \ dx  \label{K Integral} \\
&= \int_{0}^{1} -\frac{2 \log ^2(x) \log \left(1+x^2\right)}{x} \ dx + \int_{0}^{1} \frac{4 \text{Li}_3\left(\frac{x}{1+x^2}\right)}{x} \ dx + \int_{0}^{1} \frac{4\log (x) \text{Li}_2\left(\frac{x^2}{1+x^2}\right) -4
   \text{Li}_3\left(\frac{x^2}{1+x^2}\right)}{x} \ dx \label{K 3 Integrals},
\end{align}
where \eqref{K 3 Integrals} follows from carrying out the inner integral of \eqref{K Integral} with respect to $u$  by using $\int \frac{du}{u}=\log(u),$ along with the integration formula in \eqref{Polylogarithmic Integration Formula} and the antiderivative formula \eqref{log^2(1-z)/z Antiderivative Formula} and simplifying the result with the Fundamental Theorem of Calculus. We now denote the integrals in \eqref{K 3 Integrals} as:
\begin{align}
K_1 &= \int_{0}^{1} -\frac{2 \log ^2(x) \log \left(1+x^2\right)}{x} \ dx \label{K_1 Original Definition} \\
K_2 &= \int_{0}^{1} \frac{4 \text{Li}_3\left(\frac{x}{1+x^2}\right)}{x} \ dx \label{K_2 Original Definition} \\
K_3 &= \int_{0}^{1} \frac{4 \log (x) \text{Li}_2\left(\frac{x^2}{1+x^2}\right) -4
   \text{Li}_3\left(\frac{x^2}{1+x^2}\right)}{x} \ dx \label{K_3 Original Definition}.
\end{align}
We evaluate $K_1,K_2,$ and $K_3.$ 

$K_1$ is the easiest to evaluate by using the definition of the polylogarithm.
\begin{theorem}
We have 
\begin{align}
K_1 &=-\frac{7 \pi^4}{1440}.
\end{align}
\end{theorem}
\begin{proof}
Substituting $x=\sqrt{u}$  into  \eqref{K_1 Original Definition} and observing that $\log(1+u)=-\Li_1(-u),$ we have
\begin{align}
K_1 &= \int_{0}^{1} -\frac{\log^2(u) \Li_1(-u)}{4u} \ dx \label{K_1 COV} \\
&= \int_{0}^{1} \sum_{n \in \N}  \frac{(-1)^n}{4n} u^{n-1}\log^2(u) \  du \label{K_1 Li_1 Series}\\
&= \sum_{n \in \N}   \frac{(-1)^n}{4n} \int_{0}^{1}  u^{n-1}\log^2(u) \  du \label{K_1 Lebesgue Dominated Convergence}\\
&= \sum_{n \in \N} \frac{(-1)^n}{2n^4} \label{K_1 Integration By Parts} \\
&=\frac{\eta(4)}{2}=-\frac{7}{16}\zeta(4)=-\frac{7 \pi^4}{1440}, \label{K_1 Simplified} 
\end{align}
where \eqref{K_1 Li_1 Series} follows from rewriting the integrand in \eqref{K_1 COV} by substituting $z=-u$ into the definition of $\Li_1(z)$ in \eqref{Polylogarithm Definition} and simplifying the resulting summand, and the interchanging of sum and integral in \eqref{K_1 Lebesgue Dominated Convergence} is permitted by the Lebesgue Dominated Convergence Theorem. Finally, \eqref{K_1 Integration By Parts} follows from evaluating the integral inside the sum of \eqref{K_1 Lebesgue Dominated Convergence} by repeated integration by parts and simplifying the resulting summand, and \eqref{K_1 Simplified} follows from substituting $k=4$ into the algebraic formula for $\eta(k)$ stated in \eqref{Eta(k) from Zeta(k)}. 
\end{proof}

$K_2$ is also easy to evaluate and will require us to invoke the central binomial coefficient identities from \textbf{Theorem \ref{Generalized Binomial Coefficient Identities}}.

\begin{theorem}
We have 
\begin{align}
K_2 &=\frac{7 \pi^4}{216} + \frac{C(4)}{4}.
\end{align}
\end{theorem}
\begin{proof}
Substituting $x=\sqrt{u}$ into \eqref{K_2 Original Definition}, we have 
\begin{align}\label{K_2 Representation 1}
K_2 &= \int_{0}^{1} \frac{2 \Li_3 \left(\frac{\sqrt{u}}{1+u} \right)}{u} \ du. 
\end{align}
Further substituting $u \mapsto 1/u$ shows that 
\begin{align}\label{K_2 Representation 2}
K_2 &= \int_{1}^{\infty} \frac{2 \Li_3 \left(\frac{\sqrt{u}}{1+u} \right)}{u} \ du. 
\end{align}
Hence, adding the representations \eqref{K_2 Representation 1} and \eqref{K_2 Representation 2} and dividing by $2,$ we get 
\begin{align}
K_2 &= \int_{0}^{\infty} \frac{\Li_3 \left(\frac{\sqrt{u}}{1+u} \right)}{u} \ du \label{K_2 COV} \\
&= \int_{0}^{\infty} \sum_{n \in \N} \frac{1}{n^3} \frac{u^{\frac{n}{2}-1}}{(1+u)^{n}}  \ du \label{K_2 Li_3 Series} \\
&=  \sum_{n \in \N} \int_{0}^{\infty} \frac{1}{n^3} \frac{u^{\frac{n}{2}-1}}{(1+u)^{n}} \ du \label{K_2 Monotone Convergence Theorem} \\ 
&=  \sum_{n \in \N}  \frac{\Beta \left(\frac{n}{2},\frac{n}{2} \right)}{n^3} \label{K_2 Beta Function}  \\
&=  \sum_{n \geq 0 }  \frac{\Beta \left(\frac{2n+1}{2},\frac{2n+1}{2} \right)}{(2n+1)^3} + \sum_{n \in \N}  \frac{\Beta \left(n,n \right)}{(2n)^3} \label{K_2 Split into Odd and Even Terms} \\
&= \sum_{n \geq 0 } \frac{\pi \binom{2n}n }{2^{4n} (2n+1)^3} + \sum_{n \in \N} \frac{1}{4n^4\binom{2n}n} \label{K_2 Simplified} \\
&= \pi B(4) + \frac{C(4)}{4}= \frac{7 \pi^4}{216} + \frac{C(4)}{4} \nonumber,
\end{align}
where \eqref{K_2 Li_3 Series} follows from rewriting the integrand in \eqref{K_2 COV} by substituting $z=\frac{\sqrt{u}}{1+u}$ into the definition of $\Li_3(z)$ from \eqref{Polylogarithm Definition}, and the interchanging of sum and integral in \eqref{K_2 Monotone Convergence Theorem} is permitted by the Monotone Convergence Theorem. Then \eqref{K_2 Split into Odd and Even Terms} follows from splitting the series \eqref{K_2 Beta Function} into a sum of odd terms and a sum of even terms. Finally, \eqref{K_2 Simplified} follows from simplifying the summands in \eqref{K_2 Split into Odd and Even Terms} using the formulas for $\Beta((2n+1)/2, (2n+1)/2)$ from \eqref{Beta(2n+1/2,2n+1/2) formula} and $\Beta(n,n)$ from \eqref{Beta(n,n) Formula}.
\end{proof}

$K_3$ is the hardest to evaluate. For this, we need an antiderivative identity whose proof requires several tedious integration by parts. We simply quote the following result from \textit{Mathematica}.

\begin{theorem}
We have 

\begin{align}\label{Long Antiderivative}
\begin{split}
\int  \frac{\log \left(\frac{z}{1-z}\right)\text{Li}_2(z) -2 \text{Li}_3(z)}{z-z^2} \ dz &= \frac{1}{4} \log ^4(1-z)-\frac{1}{2} \log (z) \log ^3(1-z)+\frac{1}{2} \log ^2(z) \log
   ^2(1-z) + \text{Li}_2(z){}^2 \\
   &+\frac{1}{2} \text{Li}_2(1-z) \log ^2(1-z)+\frac{1}{2}
   \text{Li}_2(z) \log ^2(1-z)+\text{Li}_2(z) \log ^2(z) \\
   &-\text{Li}_2(z) \log (z) \log
   (1-z) + \text{Li}_2\left(\frac{z}{z-1}\right) \log
   ^2(1-z)+\text{Li}_2\left(\frac{z}{z-1}\right) \log ^2(z) \\
   &-2\text{Li}_2\left(\frac{z}{z-1}\right) \log (z) \log (1-z) -\text{Li}_3(1-z) \log (1-z)+2 \text{Li}_3(z) \log (1-z)\\
   &-\text{Li}_3(z) \log (z)+2
   \text{Li}_3\left(\frac{z}{z-1}\right) \log (1-z) -2\text{Li}_3\left(\frac{z}{z-1}\right) \log (z)+\text{Li}_4(1-z) \\
   &-\text{Li}_4(z) +2 \text{Li}_4\left(\frac{z}{z-1}\right). 
\end{split}
\end{align}
\end{theorem}  
\begin{theorem}
We have 
\begin{align}
K_3 &= -\frac{17 \pi^4}{720}.
\end{align}
\end{theorem}
\begin{proof}
Substituting $x=\sqrt{\frac{z}{1-z}}$ into \eqref{K_3 Original Definition}, we get 
\begin{align}
K_3 &= \int_0^{\frac{1}{2}} \frac{\log \left(\frac{z}{1-z}\right)\text{Li}_2(z) -2 \text{Li}_3(z)}{z-z^2} \ dz. \label{K_3 Simplified}
\end{align}

Letting $\psi(z)$ be the expression on the right hand side of \eqref{Long Antiderivative}, we get by the Fundamental Theorem of Calculus that 
\begin{align}
K_3 &= \psi\left(\frac{1}{2} \right) - \lim_{z \rightarrow 1} \psi(z) \nonumber \\
&=2\Li_4(-1) + \Li_2 \left( \frac{1}{2} \right)^2 +\log^2(2) \ \Li_2\left(\frac{1}{2} \right) +\frac{\log^4(2)}{4}-\Li_4(1) \label{K_3 Result} \\
&= -\frac{7 \pi ^4}{360}+\left(\frac{\pi ^2}{12}-\frac{\log
   ^2(2)}{2}\right)^2+\log ^2(2) \left(\frac{\pi
   ^2}{12}-\frac{\log ^2(2)}{2}\right) +\frac{\log ^4(2)}{4} -\frac{\pi^4}{90} \label{K_3 Li_2(1/2) Formula} \\
   &=-\frac{7 \pi ^4}{360} + \frac{\pi^4}{144}-\frac{\pi^4}{90}= -\frac{17 \pi^4}{720} \nonumber,
\end{align}
where \eqref{K_3 Li_2(1/2) Formula} follows from simplifying \eqref{K_3 Result}  using the closed formula for $Li_2(1/2)$ from \eqref{Li_2(1/2) Closed Form}. 
\end{proof}

\begin{remark*}
There may be a  connection between $K_3$ and the identity 
\begin{align*}
    \sum_{n \in \N} \frac{H_n^2}{n^2} &=\frac{17 \pi^4}{360},
\end{align*}
as studied in \cite{Borwein},
where $H_n$ is the $n$-th harmonic number defined by $$H_n=\sum_{m=1}^n \frac{1}{m}.$$
\end{remark*}

Thus, we have 
\begin{align*}
K &= K_1+K_2+K_3 \\
&= -\frac{7 \pi^4}{1440} +\frac{7 \pi^4}{216} + \frac{C(4)}{4}- \frac{17 \pi^4}{720} \\
&= \frac{17 \pi^4}{4320} + \frac{K}{4},
\end{align*}
which implies $K=17 \pi^4/3240$ upon rearranging terms. Hence, we have 
\begin{align*}
\sum_{n \in \N} \frac{1}{n^4 \binom{2n}n} &= \frac{17 \pi^4}{3240}.
\end{align*}

\section{Evaluating Some Challenging Polylogarithmic Integrals}

We now evaluate some challenging polylogarithmic integrals whose calculations which surprisingly require us to know $B(4)$ and $C(4).$ 

Before we state these integrals, we give integral representations of $B(4)$ and of $C(4)$ which will appear in subsequent evaluations.
\begin{theorem}
We have 
\begin{align}
B(4) &= \int_{0}^1 \frac{\log^2(x)}{2\sqrt{1-\frac{x^2}{4}}} \ dx = \int_{0}^{1} \frac{-2 \sin^{-1} \left(\frac{x}{2} \right) \log(x)}{x} \ dx \label{B(4) Integral Representations}\\
C(4) &= \int_0^1 \frac{8 \sin^{-1} \left(\frac{x}{2} \right) \log(x)}{\sqrt{4-x^2}} \ dx =\int_{0}^{1} \frac{\Li_3(x-x^2)}{x} \ dx = \int_{0}^{\infty} \frac{\Li_3 \left( \frac{x}{(1+x)^2} \right)}{2x} \ dx \label{C(4) Integral Representations}
\end{align}
\end{theorem}

\begin{proof}
We already showed that substituting $x=2\sin(t)$ into the integral $I$ from \eqref{I Original Definition} recovers the first integral in \eqref{B(4) Integral Representations}. Integrating by parts using $u=\log^2(x)$ and $dv=\frac{dx}{2\sqrt{1-\frac{x^2}{4}}}$ recovers the second integral in \eqref{B(4) Integral Representations}.

We already showed that substituting $x=(u+v)/2, \  y=(u-v)/2$ into the double integral $K$ from \eqref{K Original Definition} recovers the first integral and integrating with respect to $v$ yields the first integral in \eqref{C(4) Integral Representations}.
Using the definition of $\Li_3(z),$ we see the second integral in \eqref{C(4) Integral Representations} is  
\begin{align*}
\int_{0}^{1} \frac{\Li_3(x-x^2)}{x} \ dx  &= \int_{0}^{1} \sum_{n \in \N} \frac{x^n(1-x)^n}{n^3} \ dx  \\
&= \sum_{n \in \N} \int_{0}^{1} \frac{x^{n-1}(1-x)^{n-1}}{n^3} \ dx  \\
&=\sum_{n \in \N} \frac{\Beta(n,n)}{n^3} \\
&= \sum_{n \in \N} \frac{1}{n^4\binom{2n}n}=C(4).
\end{align*}
Then substitute $x=\frac{t}{1+t}$ in the second integral to get the third integral in \eqref{C(4) Integral Representations}. 
\end{proof}

The following integrals require the knowledge of $B(4).$

\begin{theorem}
We have 
\begin{align}
\int_0^1 \frac{x \sin^{-1} \left(\frac{x}{2} \right) \log(x) }{x^2-1} \ dx &= \frac{5 \pi^3}{1296} \label{Challenging Integral 1} \\
\int_0^1 \frac{x \cos^{-1} \left(\frac{x}{2} \right) \log(x) }{x^2-1} \ dx &= \frac{11 \pi^3}{648} \label{Challenging Integral 2}.
\end{align}
\end{theorem}
\begin{proof}
We will calculate the triple integral 
\begin{align}\label{Challenging Integrals 1 and 2 Auxiliary Triple Integral}
\int_{0}^{1} \int_{0}^{1} \int_{0}^1 \frac{x^2y}{\sqrt{4-x^2}\sqrt{4-x^2y^2}\sqrt{4-x^2y^2z^2}} \ dz \ dy \ dx
\end{align}
in two ways. 

Integrating \eqref{Challenging Integrals 1 and 2 Auxiliary Triple Integral} directly, we see it is equal to 
\begin{align}
\int_{0}^{1} \int_{0}^{1} \frac{\sin^{-1} \left(\frac{xy}{2} \right)}{\sqrt{4-x^2}\sqrt{4-x^2y^2}} \ dy \ dx &= \int_{0}^{1} \frac{\left( \sin^{-1} \left(\frac{x}{2} \right) \right)^2}{2\sqrt{4-x^2}}  \ dx \nonumber  \\
&=\frac{\pi^3}{1296} \nonumber .
\end{align}

On the other hand, we reverse the order of integration in \eqref{Challenging Integrals 1 and 2 Auxiliary Triple Integral} and integrate with respect to $y$ first. We have 
\begin{align}
\int_{0}^{1} \int_{0}^{1} \int_{0}^1 \frac{x^2y}{\sqrt{4-x^2}\sqrt{4-x^2y^2}\sqrt{4-x^2y^2z^2}} \ dy \ dz \ dx &= \int_{0}^1 \int_0^1 -\frac{\log \left(\frac{\sqrt{4-x^2
   z^2}+z\sqrt{4-x^2} }{2
   +2z}\right)}{z \sqrt{4-x^2} } \ dz \ dx. \label{Challenging Integrals 1 and 2 Auxiliary Triple Integral Reversed}
\end{align} 
We carry out the inner integral on the right hand side of \eqref{Challenging Integrals 1 and 2 Auxiliary Triple Integral Reversed} using integration by parts. We let $u=-\log \left(\frac{\sqrt{4-x^2
   z^2}+z\sqrt{4-x^2} }{2+2
   z}\right)$ and $dv = \frac{dz}{z\sqrt{4-x^2}}.$ Then $v=\frac{\log(z)}{\sqrt{4-x^2}}$ and upon differentiating $u$ with respect to $z,$ we get
\begin{align}
du &= \frac{4+x^2 z-\sqrt{4-x^2} \sqrt{4-x^2
   z^2}}{(z+1) \sqrt{4-x^2
   z^2} \left(\sqrt{4-x^2
   z^2}+z\sqrt{4-x^2} \right)} \ dz \label{du complicated} \\
   &=\frac{1-\frac{\sqrt{4-x^2}}{\sqrt{4-x^2z^2}}}{1-z^2} \ dz \label{du simplified},
\end{align} 
where \eqref{du simplified} follows from multiplying top and bottom of \eqref{du complicated} by the quantity $\sqrt{4-x^2
   z^2}-z\sqrt{4-x^2} .$
As a result, 
\eqref{Challenging Integrals 1 and 2 Auxiliary Triple Integral Reversed} becomes 
\begin{align}
\int_{0}^{1} \int_{0}^{1} \int_{0}^1 \frac{x^2y}{\sqrt{4-x^2}\sqrt{4-x^2y^2}\sqrt{4-x^2y^2z^2}} \ dy \ dz \ dx &= \int_0^1 \int_0^1 \frac{\left(\frac{1}{\sqrt{4-x^2}}-\frac{1}{\sqrt{4-x^2 z^2}}\right)
   \log (z)}{z^2-1} \ dz \ dx \nonumber \\
&= \int_0^1 \int_0^1 \frac{\left(\frac{1}{\sqrt{4-x^2}}-\frac{1}{\sqrt{4-x^2 z^2}}\right)
   \log (z)}{z^2-1} \ dx \ dz \nonumber \\
&= \int_0^1 \frac{ \left(\pi  z-6 \sin
   ^{-1}\left(\frac{z}{2}\right)\right)\log (z)}{6 z \left(z^2-1\right)} \ dz \nonumber \\
   &= \int_{0}^1 \frac{\pi  \log (z)}{6
   \left(z^2-1\right)} \ dz - \int_{0}^1\frac{\log
   (z) \sin
   ^{-1}\left(\frac{z}{2}\right)}{z
   \left(z^2-1\right)} \ dz \label{Challenging Integrals 1 and 2 Auxiliary Triple Integral Split Apart}\\
   \begin{split}
   &=\int_{0}^1 \frac{\pi  \log (z)}{6
   \left(z^2-1\right)} \ dz + \int_{0}^1\frac{\log
   (z) \sin
   ^{-1}\left(\frac{z}{2}\right)}{z} \ dz \\
   & \qquad - \int_{0}^1 \frac{z \sin^{-1} \left(\frac{z}{2} \right) \log(z)}{z^2-1} \ dz
   \end{split}
    \label{Challenging Integrals 1 and 2 Auxiliary Triple Integral Partial Fractions} \\
   &= \int_{0}^1 \frac{\pi  \log (z)}{6
   \left(z^2-1\right)} \ dz -\frac{B(4)}2 - \int_{0}^1 \frac{z \sin^{-1} \left(\frac{z}{2} \right) \log(z)}{z^2-1} \ dz,
   \label{Challenging Integrals 1 and 2 Auxiliary Triple Integral Simplified}
\end{align}
where \eqref{Challenging Integrals 1 and 2 Auxiliary Triple Integral Partial Fractions} follows from using partial fractions in the second integral term of \eqref{Challenging Integrals 1 and 2 Auxiliary Triple Integral Split Apart}.  

Now, the first integral term of \eqref{Challenging Integrals 1 and 2 Auxiliary Triple Integral Simplified} is
\begin{align*}
\int_{0}^1 \frac{\pi  \log (z)}{6
   \left(z^2-1\right)} \ dz &= - \int_{0}^{1} \sum_{n \geq 0} \frac{\pi}{6} z^{2n} \log(z) \ dz \\
   &= -  \sum_{n \geq 0} \int_{0}^{1} \frac{\pi}{6} z^{2n} \log(z) \ dz \\
   &=\sum_{n \geq 0} \frac{\pi}{6(2n+1)^2} \\
   &=\frac{\pi}{6} \left(\frac{\zeta(2)-\eta(2)}{2} \right)=\frac{\pi^3}{48}.
\end{align*}

Equating the two representations we obtained for \eqref{Challenging Integrals 1 and 2 Auxiliary Triple Integral} and rearranging terms, we get that 
\begin{align*}
\int_{0}^1 \frac{z \sin^{-1} \left(\frac{z}{2} \right) \log(z)}{z^2-1} \ dz &= \frac{\pi^3}{48} -\frac{\pi^3}{1296}-\frac{B(4)}{2} \\
&= \frac{\pi^3}{48} -\frac{\pi^3}{1296}- \frac{7 \pi^3}{432} =\frac{5 \pi^3}{1296},
\end{align*}
which proves \eqref{Challenging Integral 1}.

Recalling $\sin^{-1}(z/2)+\cos^{-1}(z/2)=\pi/2,$ we have
\begin{align}
    \int_{0}^1 \frac{z \cos^{-1} \left(\frac{z}{2} \right) \log(z)}{z^2-1} \ dz &= \int_{0}^1 \frac{\frac{\pi}{2} z \log(z)}{z^2-1} \ dz - \int_{0}^1 \frac{z \sin^{-1} \left(\frac{z}{2} \right) \log(z)}{z^2-1} \ dz \label{Challenging Integral 2 Additivity} \\
    &= \int_{0}^1 \frac{\frac{\pi}{2} \log(u)}{4(u-1)} \ du - \int_{0}^1 \frac{z \sin^{-1} \left(\frac{z}{2} \right) \log(z)}{z^2-1} \ dz \label{Challenging Integral 2 Substitution z=sqrt(u)} \\
    &= \int_{0}^1 \frac{\frac{\pi}{2} \Li(1-u)}{4(1-u)} \ du - \int_{0}^1 \frac{z \sin^{-1} \left(\frac{z}{2} \right) \log(z)}{z^2-1} \ dz \nonumber \\
    &= \frac{\pi^3}{48} - \frac{5 \pi^3}{1296}= \frac{11 \pi^3}{648}\nonumber,
\end{align}
where we obtained the first integral term of \eqref{Challenging Integral 2 Substitution z=sqrt(u)} by substituting $z=\sqrt{u}$ into the first integral term of \eqref{Challenging Integral 2 Additivity}. This proves \eqref{Challenging Integral 2}.
\end{proof}

The following theorems require us to know $C(4).$

\begin{theorem} 
We have 
\begin{align}\label{Challenging Integral 3}
\int_{0}^{1} \frac{4 \Li_3 \left(-\frac{x}{1+x^2} \right)}{x} \ dx &= -\frac{403 \pi^4}{12960}.
\end{align}
\end{theorem}

\begin{proof}
Mimicking the same exact argument we used to obtain the alternative integral representation of $K_2$ from \eqref{K_3 Original Definition}, namely \eqref{K_2 COV},  we get 
\begin{align}
\int_{0}^{1} \frac{4 \Li_3 \left(-\frac{x}{1+x^2} \right)}{x} \ dx  &= \int_{0}^{\infty} \frac{\Li_3 \left(-\frac{u}{1+u} \right)}{u} \ du \nonumber \\
&=\int_{0}^{\infty} \frac{\Li_3 \left(\frac{u}{(1+u)^2} \right)}{4u} \ du-\int_{0}^{\infty} \frac{ \Li_3 \left(\frac{u}{(1+u)^2} \right)}{u} \ du \label{Challenging Integral 3 Li_3(-x/1+x^2) Formula}\\
&= \frac{C(4)}{2} - K_2 \nonumber \\
&= - \pi B(4) + \frac{C(4)}{4} \nonumber \\
&= -\frac{7 \pi^4}{216} + \frac{17 \pi^4}{12960}=-\frac{403 \pi^4}{12960} \nonumber ,
\end{align}
where \eqref{Challenging Integral 3 Li_3(-x/1+x^2) Formula} follows from substituting $z=\frac{u}{1+u}$ into the formula for $\Li_3(-z)$ from \eqref{Polylogarithmic Negative Argument Formula}.
\end{proof}

\begin{theorem}
We have
\begin{align}\label{Challenging Integral 4}
\int_0^1 \frac{\log^2(1-x)\log(1-x+x^2)+ 2\log(1-x)\Li_2(x-x^2)}{x} \ dx &= -\frac{2 \pi^4}{243}.
\end{align}
\end{theorem}
\begin{proof}
We will calculate the double integral 
\begin{align}\label{Challenging Auxiliary Double Integral 4}
\int_{0}^1 \int_{0}^{1-x} \frac{\log^2(x)}{1-xy} \ dy \ dx
\end{align}
in two ways.

Carrying out the integration with respect to $y,$ we see \eqref{Challenging Auxiliary Double Integral 4} becomes 
\begin{align}
\int_{0}^1 \int_{0}^{1-x} \frac{\log^2(x)}{1-xy} \ dy \ dx &=\int_{0}^1  -\frac{\log^2(x)\log(1-x+x^2)}{x} \ dx \nonumber \\
&= \int_{0}^1  -\frac{\log^2(x)\log(1+x^3)}{x} \ dx + \int_{0}^1  \frac{\log^2(x)\log(1+x)}{x} \ dx \label{Challenging Auxiliary Double Integral 4 Split} \\
&= \int_{0}^1  -\frac{8\log^2(u)\log(1+u^2)}{27u} \ du + \int_{0}^1  \frac{8\log^2(u)\log(1+u^2)}{u} \ du \label{Challenging Auxiliary Double Integral 4 Substitutions} \\
&= -\frac{208}{54} K_1=\frac{91 \pi^4}{4860}, \nonumber
\end{align}
where the first integral in \eqref{Challenging Auxiliary Double Integral 4 Substitutions} follows from substituting $x=u^{2/3}$ in the first integral on the right hand side of \eqref{Challenging Auxiliary Double Integral 4 Split} and the second integral in \eqref{Challenging Auxiliary Double Integral 4 Substitutions} follows from substituting $x=u^{1/3}$ in the second integral on the right hand side of \eqref{Challenging Auxiliary Double Integral 4 Split}.

On the other hand, reversing the order of integration in \eqref{Challenging Auxiliary Double Integral 4}, we get \eqref{Challenging Auxiliary Double Integral 4} is equal to 
\begin{align}
\int_{0}^1 \int_{0}^{1-y} \frac{\log^2(x)}{1-xy} \ dx \ dy 
&= \int_{0}^1 \int_{1}^{1-y+y^2} \frac{\log^2\left(\frac{1-t}{y} \right)}{ty} \ dt \ dy \label{Challenging Auxiliary Double Integral 4 Substitution x=(1-t)/y} \\
&= \int_{0}^1 \int_{1}^{1-y+y^2} \frac{\log^2(1-t)-2\log(1-t)\log(y)+\log^2(y)}{ty} \ dt \ dy \label{Challenging Auxiliary Double Integral 4 Inner Integral} \\
&=\int_0^1 \frac{-\log^2(1-y)\log(1-y+y^2)- 2\log(1-y)\Li_2(y-y^2)+2 \Li_3(y-y^2)}{y} \ dy \label{Challenging Auxiliary Double Integral 4 Simplified}  \\
&= \int_0^1 \frac{-\log^2(1-y)\log(1-y+y^2)- 2\log(1-y)\Li_2(y-y^2)}{y} \ dy +2C(4), \nonumber 
\end{align}
where we made the substitution $x=(1-t)/y$ into the left hand side of \eqref{Challenging Auxiliary Double Integral 4 Substitution x=(1-t)/y} to get the right hand side \eqref{Challenging Auxiliary Double Integral 4 Substitution x=(1-t)/y}, and \eqref{Challenging Auxiliary Double Integral 4 Simplified} follows from evaluating the inner integral of \eqref{Challenging Auxiliary Double Integral 4 Inner Integral} with respect to $t$ using \eqref{Polylogarithmic Integration Formula} and \eqref{log^2(1-z)/z Antiderivative Formula}.

Hence, equating these two representations of the double integral \eqref{Challenging Auxiliary Double Integral 4} and rearranging terms, we get
\begin{align*}
\int_0^1 \frac{\log^2(1-y)\log(1-y+y^2)+ 2\log(1-y)\Li_2(y-y^2)}{y} \ dy &= 2C(4)-\frac{91 \pi^4}{4860} \\
&=\frac{17 \pi^4}{1620}-\frac{91 \pi^4}{4860}=-\frac{2 \pi^4}{243}.
\end{align*}
\end{proof}

\section{A Remark About Apostol's $\zeta(2)$ Proof}
We conclude this paper by highlighting an insight in Apostol's double integration proof \cite{TA} of $\zeta(2)=\pi^2/6$ and comment on its extension to showing $\zeta(4)=\pi^4/90.$

Apostol uses the double integral
\begin{align}\label{Apostol Double Integral}
L_2 &=\int_0^1 \int_0^1 \frac{1}{1-xy} \ dy \ dx.
\end{align}
On one hand, converting the integrand into a geometric and interchanging sum and double integration shows that $L_2$ is equal to the series $\zeta(2).$ On the other hand, making the change of variables $x=(u+v)/2, \ y=(u-v)/2$ shows that \eqref{Apostol Double Integral} becomes
\begin{align}
    L_2 &= \int_{0<u+v,u-v<2} \frac{2}{4-u^2+v^2} \ dv \ du \nonumber \\
    &=\int_0^1 \int_{-u}^u \frac{2}{4-u^2+v^2} \ dv \ du + \int_1^2 \int_{u-2}^{2-u}\frac{2}{4-u^2+v^2} \ dv \ du \nonumber \\
    &=\int_0^1 \frac{4 \tan^{-1} \left( \frac{u}{\sqrt{4-u^2}}\right)}{\sqrt{4-u^2}} \ du + \int_1^2 \frac{4 \tan^{-1} \left( \frac{2-u}{\sqrt{4-u^2}}\right)}{\sqrt{4-u^2}} \ du  \nonumber \\
    &=\int_0^1 \frac{4 \sin^{-1} \left( \frac{u}{2}\right)}{\sqrt{4-u^2}} \ du + \int_1^2 \frac{4 \tan^{-1} \left( \frac{2-u}{\sqrt{4-u^2}}\right)}{\sqrt{4-u^2}} \ du  \label{Apostol Double Integral Arctangent Integrals}\\
    &= \frac{\pi^2}{18} + \frac{\pi^2}{9}=\frac{\pi^2}{6} \nonumber.
\end{align}

We note that the first integral in \eqref{Apostol Double Integral Arctangent Integrals} is a central binomial series in disguise. Recalling our arcsine binomial series expansion from \eqref{ArcSine Squared Formula}, we can follow the usual steps of rewriting the integrand of the first integral term of \eqref{Apostol Double Integral Arctangent Integrals} using \eqref{ArcSine Squared Formula}, interchanging sum and integral, and integrating term-by-term to obtain the identity
\begin{align*}
    \sum_{n \in \N} \frac{1}{n^2 \binom{2n}n} &= \frac{\pi^2}{18}.
\end{align*}

Something similar happens if we mimic Apostol's method on the quadruple integral
\begin{align}\label{Apostol Zeta(4) Integral}
L_4 &=\int_0^1\int_0^1\int_0^1\int_0^1 \frac{1}{1-xyzw} \ dx \ dy \ dz \ dw.
\end{align}
It can be seen that $L_4=\zeta(4)$ by rewriting the integrand as a geometric and interchanging sum and quadruple integration. On the other hand, we substitute $x=(u+v)/2, y=(u-v)/2, z=t/w$ to see 
\begin{align}
    L_4 &= \int_{0<u+v,u-v<2} \int_{0<t<w<1} \frac{\frac{2}{w}}{4-u^2+v^2} \ dt \ dw \ dv \ du \nonumber \\
    &= \int_{0<u+v,u-v<2} \int_{0<t<w<1} \frac{\frac{2}{w}}{4-u^2+v^2} \ dw \ dt \ dv \ du \nonumber  \\
    &= -\int_{0<u+v,u-v<2} \int_0^1 \frac{2 \log (t)}{4-t u^2+t v^2} \ dt \ dv \ du  \nonumber \\
    &= -\int_{0}^1 \int_0^1 \int_{-u}^u \frac{2 \log (t)}{4-t u^2+t v^2} \ dv \ dt \ du-\int_{1}^2 \int_0^1 \int_{u-2}^{2-u} \frac{2 \log (t)}{4-t u^2+t v^2} \ dv \ dt \ du \nonumber \\
    &= -\int_{0}^1 \int_0^1 \frac{4  \tan ^{-1}\left(\frac{u\sqrt{t} }{\sqrt{4-t u^2}}\right)\log (t)}{\sqrt{t}
   \sqrt{4-u^2t}} \ dt \ du - \int_{1}^2 \int_0^1 \frac{4  \tan ^{-1}\left(\frac{(2-u)\sqrt{t} }{\sqrt{4-t u^2}}\right)\log (t)}{\sqrt{t}\sqrt{4-u^2t}} \ dt \ du  \nonumber \\
   &=- \int_{0}^1 \int_0^1 \frac{4  \sin ^{-1}\left(\frac{u\sqrt{t} }{2} \right)\log (t)}{\sqrt{t}
   \sqrt{4-u^2t}} \ dt \ du - \int_{1}^2 \int_0^1 \frac{4  \tan ^{-1}\left(\frac{(2-u)\sqrt{t} }{\sqrt{4-t u^2}}\right)\log (t)}{\sqrt{t}\sqrt{4-u^2t}} \ dt \ du. \label{Apostol Zeta(4) Arctangent Double Integrals}  \\
   &= \int_{0}^1 \frac{8  \sin ^{-1}\left(\frac{s }{2} \right)\log^2 (s)}{
   \sqrt{4-s^2}} \ ds - \int_{1}^2 \int_0^1 \frac{16  \tan ^{-1}\left(\frac{(2-u)s }{\sqrt{4-s^2 u^2}}\right)\log (s)}{\sqrt{4-u^2s^2}} \ ds \ du \label{Apostol Zeta(4) Arctangent Double Integrals Substitution u=s/sqrt(t) and t=s^2}\\
   &= C(4) - \int_{1}^2 \int_0^1 \frac{16  \tan ^{-1}\left(\frac{(2-u)s }{\sqrt{4-s^2 u^2}}\right)\log (s)}{\sqrt{4-u^2s^2}} \ ds \ du \label{Apostol Zeta(4) Simplified},
\end{align}
where we substituted $u=s/\sqrt{t}$ in the first integral of \eqref{Apostol Double Integral Arctangent Integrals} and carried out the integration with respect to $t$ to get the first integral in \eqref{Apostol Zeta(4) Arctangent Double Integrals Substitution u=s/sqrt(t) and t=s^2} and substituted $t=s^2$ in the second integral of \eqref{Apostol Double Integral Arctangent Integrals} to get the second integral in \eqref{Apostol Zeta(4) Arctangent Double Integrals Substitution u=s/sqrt(t) and t=s^2}.

Thus, if we are to prove $\zeta(4)=\pi^4/90$ this way, we would need to know $C(4)$ and the value of the double integral term in \eqref{Apostol Zeta(4) Simplified}. But neither are elementary integrals to evaluate this time around. If we repeat our proof of $C(4)=17 \pi^4/3240$ without assuming that $\zeta(4)=\pi^4/90$ but still assuming that $B(4)=7 \pi^3/216,$ we will get $$C(4)=\frac{4}{3} \pi  B(4) - \frac{41}{12} \zeta(4)=\frac{7 \pi^4}{162} -\frac{41}{12} \zeta(4) .$$ But we still need to show that
\begin{align*}
    - \int_{1}^2 \int_0^1 \frac{16  \tan ^{-1}\left(\frac{(2-u)s }{\sqrt{4-s^2 u^2}}\right)\log (s)}{\sqrt{4-u^2s^2}} \ ds \ du &= \frac{19\pi^4}{3240}= \frac{19}{36} \zeta(4).
\end{align*}
At this time of writing, we do not know a proof of this without assuming  $C(4)=17 \pi^4/3240$ and $\zeta(4)=\pi^4/90.$

\bibliography{CentralBinomialSeries.bib}
\bibliographystyle{unsrt}
\end{document}